\documentclass[12pt]{article}

\usepackage[a4paper, total={7in, 9in}]{geometry}

\usepackage[english]{babel}
\usepackage[utf8]{inputenc}
\usepackage{amsmath}
\usepackage{graphicx}
\usepackage[colorinlistoftodos]{todonotes}
\usepackage{amssymb}
\usepackage{amsthm}
\usepackage{multicol}

\usepackage{mathrsfs}
\usepackage{enumerate}
\usepackage{amsfonts}
\usepackage{verbatim}
\usepackage{amsbsy}
\usepackage{amsmath}
\usepackage{amssymb}
\usepackage{MnSymbol}
\usepackage{mathrsfs}
\usepackage{xcolor}
\usepackage{cite}

\newtheorem{dfn}{Definition}[section]
\newtheorem{thm}{Theorem}[section]
\newtheorem{prop}{Proposition}[section]
\newtheorem{lem}{Lemma}[section]
\newtheorem{rem}{Remark}[section]
\newtheorem{eg}{Example}[section]
\newtheorem{cor}{Corollary}[section]
\newtheorem{assp}{Assumption}[section]

\DeclareMathOperator{\ot}{OT}
\DeclareMathOperator{\dom}{Dom}
\DeclareMathOperator{\id}{id}
\DeclareMathOperator{\fr}{FR}

\DeclareMathOperator{\im}{Im}

\DeclareMathOperator{\diag}{diag}

\title{A Ramsey Algebraic Study of Matrices}
\author{Zu Yao Teoh \& Wen Chean Teh\footnote{Corresponding author.}}

\date{April 27, 2017}

\begin{document}
\maketitle

\abstract{The notion of a topological Ramsey space was introduced by Carlson some 30 years ago. Studying the topological Ramsey space of variable words, Carlson was able to derive many classical combinatorial results in a unifying manner. For the class of spaces generated by algebras, Carlson had suggested that one should attempt a purely combinatorial approach to the study. This approach was later formulated and named Ramsey algebra. In this paper, we continue to look at heterogeneous Ramsey algebras, mainly characterizing various Ramsey algebras involving matrices.}

\section{Introduction}
The notion of a Ramsey algebra came as an offshoot of Carlson's pioneering work on (topological\footnote{What is known as a topological Ramsey space in the modern literature is known simply as a Ramsey space in Carlson's original work. The adjective ``topological" is added through Todorocevic's extension \cite{sT10} of Carlson's work on the subject.}) Ramsey spaces \cite{carl88} when he suggested that the class of Ramsey spaces induced by algebras can be singled out and be studied combinatorially. Of particular importance was the space of multivariable words. By making the right choices of alphabets, Carlson derived a wide array of classic combinatorial results as corollaries to a result of his concerning the topological Ramsey space of multivariable words, results which were otherwise derived on independent grounds. Among those classical results were the Hales-Jewett theorem, Ellentuck's theorem, and Hindman's theorem in particular.

The initiating studies on Ramsey algebras can be found in the papers \cite{teh2016ramsey}, \cite{wcT13b}, and \cite{tehidem16}. In \cite{tehidem16}, the second author addresses a question of Carlson concerning the existence of idempotent ultrafilters for Ramsey algebras. A precise connection between the notion of a (topological) Ramsey space and the notion of a Ramsey algebra can be found in Section 4 of \cite{teohteh}. As the name implies, Ramsey algebras are algebras possessing a certain homogeneity property as is the case with any Ramsey-type result. We will give a precise definition for what is meant by a (heterogeneous) Ramsey algebra in the next section.

An algebra consists of a nonempty domain and a collection of operations on the domain. We will view each algebra as a model of a many-sorted first-order language whose members (the nonlogical symbols) consist of function symbols. In this paper, we will look at the Ramsey algebraic aspects of various reducts of the matrix algebra consisting of matrix addition and multiplication, the field operations, and the determinant operation. We will study generalizations to wider classes of algebras and derive the properties concerning matrices as corollaries.

\section{Preliminaries}
The set of natural numbers will be denoted by $\omega$ and natural numbers include $0$. The positive integers will be denoted by $\mathbb{N}$.

Let $\{A_\xi\}_{\xi\in I}$ be a family of nonempty sets with $I$ the indexing set and let $\mathcal{F}$ be a family of operations on $\{A_\xi\}_{\xi\in I}$. A function $f$ is said to be an \emph{operation} on $\{A_\xi\}_{\xi\in I}$ if  the domain of $f$ equals $\prod_{\xi\in J}A_\xi$ for some finite $J\subseteq I$ and the codomain of $f$ is $A_\xi$ for some $\xi\in I$. The structure $(\{A_\xi\}_{\xi\in I}, \mathcal{F})$ is called a \emph{heterogeneous algebra} or \emph{algebra} in short. If $I$ is a singleton, the algebra is referred to as a \emph{homogeneous algebra} for emphasis. The family $\{A_\xi\}_{\xi\in I}$ is called the \emph{domain} of the algebra and each member of the family a \emph{phylum}. Every $\vec{e}\in{^\omega}I$ is called a sort and, if $\vec{b}\in{^\omega}\left(\bigcup_{\xi\in I}A_i\right)$, then $\vec{b}$ is said to be $\vec{e}$-sorted if $\vec{b}(i)\in A_{\vec{e}(i)}$ for each $i\in\omega$. If the lists of phyla or operations are not long, we will write them out explicitly. For instance, if $I=\{0, 1\}$ and $\mathcal{F}=\{\circ, +, \times\}$, the algebra would be written $(A_1, A_0, \circ, +, \times)$.

Following Carlson's convention, we require that the phyla in a given algebra be pairwise disjoint. Hence:

\begin{rem}
For any algebra $(\{A_\xi\}_{\xi\in I}, \mathcal{F})$, the sort of any given $\vec{b}\in{^\omega\!\left(\bigcup_{\xi\in I}A_i\right)}$ is unique.
\end{rem}

We will call an operation $f$ \emph{heterogeneous} if $f:A_{\xi_1}\times\cdots\times A_{\xi_N}\to A_{\xi_{N+1}}$ and there exist $i, j\in\{1, \ldots, N+1\}$ such that $A_{\xi_i}\neq A_{\xi_j}$. The domain of an operation $f$ will be denoted by $\dom(f)$ and the image set $\im(f)$. The identity function on any set $A$ is denoted by $\id_A$.

For notational convenience, we will sometimes write an $n$-tuple $\bar{x}=(x_1, \ldots, x_n)$ in the notation of a sequence $\vec{x}=\langle x_1, \ldots, x_n\rangle$ so that, for instance, if $f$ is an operation whose domain is the Cartesian product of $n$ sets, then $f(\vec{x})$ will mean the same thing as $f(\bar{x})$. The concatenation operation of sequences will be denoted by $\ast$. Now, let $\mathcal{F}$ be a family of functions on $\{A_\xi\}_{\xi\in I}$. Define $\mathcal{F}_0=\mathcal{F}\cup\{\id_{A_\xi}:\xi\in I\}$ and, suppose that $\mathcal{F}_k$ has been defined, let
\begin{equation*}
\mathcal{F}_{k+1}=\mathcal{F}_k\cup\{f:\varphi(f)\},
\end{equation*}
where $\varphi(f)$ is the statement that there exist some $N$-ary operation $g\in\mathcal{F}$ and some $h_1, \ldots, h_N\in\mathcal{F}_k$ such that $f(\vec{x})=g(h_1(\vec{x}_1), \ldots, h_N(\vec{x}_N))$ and $\vec{x}_1\ast\cdots\ast\vec{x}_N=\vec{x}=\langle x_1, \ldots, x_n\rangle$, where $n$ is the arity of $f$. Then, we have the following definition:

\begin{dfn}[\bf{Orderly Term}]
Denote the set $\bigcup_{k\in\omega}\mathcal{F}_k$ by $\ot(\mathcal{F})$. Each member of $\ot(\mathcal{F})$ is called an orderly term over $\mathcal{F}$.
\end{dfn}

\begin{eg}
Consider the addition $+$ and multiplication $\times$ operations on matrices. The composition $f(x_1, x_2, x_3, x_4)=\times(+(x_1, x_2), +(x_3, x_4))=(x_1+x_2)(x_3+x_4)$ is an orderly composition over $\{+, \times\}$. Another example is $g(x_1, x_2, x_3, x_4)=\times(+(\times(x_1, x_2), x_3), x_4)=(x_1x_2+x_3)x_4$. However, the compositions $h(x_1, x_2, x_3)=+(x_2, \times(x_1, x_3))=x_2+x_1x_3$ and $k(x_1, x_2, x_3)=\times(\times(x_3, +(x_2, x_1)), x_4)=(x_3(x_2+x_1))x_4$ are not.
\end{eg}

\begin{rem}\label{unaryorderlyterm}
Every unary orderly term over a given $\mathcal{F}$ is clearly a composition of unary operations of $\mathcal{F}$; conversely, every composition of unary operations in $\mathcal{F}$ is a unary orderly term over $\mathcal{F}$.
\end{rem}

%\begin{dfn}[\bf{Finite Reduction}]
%Let $((A_\xi)_{\xi\in I}, \mathcal{F})$ be an algebra and $\vec{b}\in{^\omega\!\left(\bigcup_{\xi\in I}A_\xi\right)}$. Then $c\in\bigcup_{\xi\in I}A_\xi$ is said to be a \emph{finite reduction} of $\vec{b}$ if $c$ is the image of some finite subsequence of $\vec{b}$ under some orderly term over $\mathcal{F}$.
%\end{dfn}
%
%We say that $c$ is a finite reduction \emph{involving} the subsequence $\vec{b}_0$ of $\vec{b}$ if $c$ is the image of $\bar{b}_0$ under some orderly term.

\begin{dfn}[\bf{Reduction $\leq_\mathcal{F}$}]\label{reduction}
Let $(\{A_\xi\}_{\xi\in I}, \mathcal{F})$ be an algebra and let $\vec{a}, \vec{b}\in{^\omega}\left(\bigcup_{\xi\in I}A_\xi\right)$. Then $\vec{a}$ is said to be a \emph{reduction} of $\vec{b}$, written $\vec{a}\leq_\mathcal{F}\vec{b}$, if for each $j\in\omega$, there exist a subsequence $\vec{b}_j$ of $\vec{b}$ and an $f_j\in\ot(\mathcal{F})$ such that
\begin{enumerate}
\item $\vec{a}(j)=f_j(\vec{b}_j)$ and
\item $\vec{b}_0\ast\vec{b}_1\ast\cdots$ forms a subsequence of $\vec{b}$.
\end{enumerate}
\end{dfn}

Note that if $\vec{a}$ is a subsequence of $\vec{b}$, then $\vec{a}\leq_\mathcal{F}\vec{b}$. We will make free use of this fact throughout. Also note that $\leq_\mathcal{F}$ is a transitive relation and, if $\mathcal{G}\subseteq\mathcal{F}$ are families of operations and $\vec{a}\leq_\mathcal{G}\vec{b}$, then $\vec{a}\leq_\mathcal{F}\vec{b}$.

\begin{dfn}\label{FR}
Let $(\{A_\xi\}_{\xi\in I}, \mathcal{F})$ be an algebra and $\vec{e}\in{^\omega}I$. For each $\vec{e}$-sorted sequence $\vec{b}$, define
\begin{equation}\label{FRdef}
\fr_\mathcal{F}^{\vec{e}}(\vec{b})=\left\{\vec{a}(0):\vec{a}\leq_\mathcal{F}\vec{b}\;\text{and}\;\vec{a}\;\text{is}\;\vec{e}\text{-sorted}\right\}.
\end{equation}
\end{dfn}

We are now ready for the notion of a heterogeneous Ramsey algebra.

\begin{dfn}[\bf{$\vec{e}$-Ramsey Algebra}]\label{RA}
Suppose that $(\{A_\xi\}_{\xi\in I}, \mathcal{F})$ is an algebra and $\vec{e}\in{^\omega\!I}$. Then $(\{A_\xi\}_{\xi\in I}, \mathcal{F})$  is said to be an \emph{$\vec{e}$-Ramsey algebra} if, for each $\vec{e}$-sorted sequence $\vec{b}$ and each $X\subseteq A_{\vec{e}(0)}$, there exists an $\vec{e}$-sorted reduction $\vec{a}$ of $\vec{b}$ such that $\emph{FR}^{\vec{e}}_\mathcal{F}(\vec{a})$ is either contained in or disjoint from $X$.

Such a sequence $\vec{a}$ is said to be \emph{homogeneous} for $X$ (with respect to $\mathcal{F}$).
\end{dfn}

We now look at some examples.

\begin{thm}[Hindman]\label{semigroup}\label{hindman}
$(\mathbb{N}, +)$ is a Ramsey algebra. More generally, every semigroup is a Ramsey algebra.
\end{thm}

The next theorem can be found in \cite{teh2016ramsey} as Theorem 5.5.

\begin{thm}\label{tehring}
The following are \emph{not} Ramsey algebras:
\begin{enumerate}
\item Infinite rings without zero divisors.
\item Infinite rings of characteristic zero with identity.
\end{enumerate}
\end{thm}

Thus, for instance, neither the field of real numbers nor the ring of square matrices (for any given order) over an infinite field of characteristic zero is a Ramsey algebra. The latter fact owes itself to the fact that the underlying field can be embedded into such matrix rings as diagonal matrices.

Vector spaces are examples of heterogeneous algebras. A vector space is an algebra with two phyla $A_1, A_0$, the former of which we set to be the set of vectors and the latter the underlying scalar field; the operations are scalar multiplication, vector addition, and the addition and multiplication of scalar elements. The following is Theorem 6.1 of \cite{teohteh}.

\begin{thm}
Let $\mathcal{V}$ be a vector space. Then:
\begin{enumerate}
\item If the underlying scalar field is finite, then $\mathcal{V}$ is an $\vec{e}$-Ramsey algebra for all $\vec{e}$.
\item If the underlying field is infinite, then $\mathcal{V}$ is an $\vec{e}$-Ramsey algebra \emph{only} for sorts $\vec{e}$ that are nonconstant but eventually constant or for those that are constant with value $1$.
\end{enumerate}
\end{thm}

Let $I$ be the indexing set of some algebra. Define
\begin{equation*}
\Omega=\{\vec{e}\in{^\omega}I:\;\text{if}\;\vec{e}(i)=\xi\;\text{for some}\;i,\;\text{then}\;\vec{e}(i)=\xi\;\text{for infinitely many}\;i\}.
\end{equation*}
Further, for each $\xi\in I$, the set of all $\vec{e}\in\Omega$ such that $\vec{e}(0)=\xi$ will be denoted by $\Omega_\xi$. For sorts $\vec{e}\in\Omega$, the sets defined by Eq. \ref{FRdef} has the following characterization:
\begin{equation}\label{Omegafr}
c\in\fr_\mathcal{F}^{\vec{e}}(\vec{b})\Longleftrightarrow c=f(\tau)
\end{equation}
for some $f\in\ot(\mathcal{F})$ and some finite subsequence $\tau$ of $\vec{b}$. In particular, in the case of homogeneous algebras, the sets given by Eq. \ref{FRdef} can be characterize as
\begin{equation}\label{FRdefHomAlg}
\fr_\mathcal{F}(\vec{b})=\{f(\tau):f\in\ot(\mathcal{F}), \tau\;\text{a subsequence of}\;\vec{b}\}.
\end{equation}

We will mainly be concerned with Ramsey algebraic properties involving sorts of the class $\Omega$ since results concerning sorts of this class is more uniform. The following theorem, which appears as Theorem 5.3 of \cite{teohteh}, gives a precise formulation of this uniformity:
\begin{thm}
Let $\mathcal{A}=\left(\bigcup_{\xi\in I}A_\xi, \mathcal{F}\right)$ be an algebra, $J\subseteq I$, and define $\vec{e}\in\Omega_\eta^J$ if and only if $\vec{e}\in\Omega$, $\vec{e}(0)=\eta$, and $\{\vec{e}(i):i\in\omega\}=J$. Then $\mathcal{A}$ is an $\vec{e}$-Ramsey algebra for some $\vec{e}\in\omega$ if and only if $\mathcal{A}$ is an $\vec{e}$-Ramsey algebra for all $\vec{e}\in\omega$.
\end{thm}

Before ending this section, we mention a fact concerning subalgebras. If $\mathcal{A}=(\{A_\xi\}_{\xi\in I}$, $\mathcal{F})$ is an algebra, then a \emph{subalgebra} $\mathcal{A}'=(\{A'_\xi\}_{\xi\in I}, \mathcal{F}')$ of $\mathcal{A}$ is an algebra such that $A_\xi'\subseteq A_\xi$ for each $\xi\in I$ and, for each $f'\in\mathcal{F}'$, there exists an $f\in\mathcal{F}$ with $f:A_{\xi_1}\times\cdots\times A_{\xi_n}\to A_{\xi_{n+1}}$ such that $\displaystyle f'=f\upharpoonright\left(A_{\xi_1}'\times\cdots\times A_{\xi_n}'\right)$ (restriction property).

\begin{prop}\label{subalg}
For any sort $\vec{e}$, every subalgebra of an $\vec{e}$-Ramsey algebra is an $\vec{e}$-Ramsey algebra.
\end{prop}
\begin{proof}
Suppose $\vec{e}$ is a sort, $\mathcal{A}=(\{A_\xi\}_{\xi\in I}, \mathcal{F})$ is an $\vec{e}$-Ramsey algebra, and $\mathcal{A}'=(\{A_\xi'\}_{\xi\in I}, \mathcal{F}')$ is a subalgebra of $\mathcal{A}$. Let $X\subseteq A'_{\vec{e}(0)}$ and let an $\vec{e}$-sorted sequence $\vec{b}$ of $\bigcup_{\xi\in I}A'_\xi$ be given.

Since $A_\xi'\subseteq A_\xi$, $\vec{b}$ is also a sequence of $\bigcup_{\xi\in I}A_\xi$. Thus, pick an $\vec{e}$-sorted $\vec{a}\leq_\mathcal{F}\vec{b}$ homogeneous for $X$, i.e. $\fr_\mathcal{F}^{\vec{e}}(\vec{a})\subseteq X$ or $\fr_\mathcal{F}^{\vec{e}}(\vec{a})\subseteq A_{\vec{e}(0)}\setminus X$. Since the terms of $\vec{b}$ consist of elements of $\bigcup_{\xi\in I}A'_\xi$, the terms of $\vec{a}$ are also members of $\bigcup_{\xi\in I}A'_\xi$ by the restriction and closure properties of operations within a subalgebra, whereby $\vec{a}\leq_{\mathcal{F}'}\vec{b}$. In addition, $\fr_\mathcal{F}^{\vec{e}}(\vec{a})=\fr_{\mathcal{F}'}^{\vec{e}}(\vec{a})\subseteq A_{\vec{e}(0)}'$. Consequently, $\fr_{\mathcal{F}'}^{\vec{e}}(\vec{a})\subseteq X$ or $\fr_{\mathcal{F}'}^{\vec{e}}(\vec{a})\subseteq A_{\vec{e}(0)}'\setminus X$.
\end{proof}

\section{Ramsey-type Theorems for Various Matrix Algebras}
Throughout the paper, the set of $n\times n$ square matrices over a field $\mathbb{F}$ will be denoted by $\mathcal{M}_n(\mathbb{F})$. Addition and multiplication of matrices will be denoted by $+$ and $\times$, respectively. Addition and multiplication of field elements will come with a subscript $\mathbb{F}$. The field $\mathbb{F}$ will be assumed to be of characteristic $0$ throughout; in such a case, the rational numbers are embedded within $\mathbb{F}$, hence we will speak freely of the isomorphic copies of the integers, the natural numbers, or the rationals in the $\mathbb{F}$ in question simply as \emph{the} integers and so on. Throughout, we fix the indexing in such a way that the set $\mathbb{F}$ of scalars receives the index $0$ while the set $\mathbb{V}$ of vectors receives the index $1$.

\begin{assp}
The field $\mathbb{F}$ is assumed to be infinite with characteristic $0$ throughout this section.
\end{assp}

In this section, we begin the study of the Ramsey-algebraic properties of various matrix algebras. We call the algebra $(\mathcal{M}_n({\mathbb{F}}), \mathbb{F}, +, \times, +_{\mathbb{F}}, \times_{\mathbb{F}}, |\ast|)$ the \emph{full} matrix algebra and any reduct of it is known as \emph{a} matrix algebra. We will be studying these algebras by looking at slightly more general algebras.

\begin{assp}\label{assp2}
The algebras $\mathcal{A}=(A_0, A_1, \mathcal{F})$ studied in this section are of the form $\mathcal{F}=\mathcal{G}_0\cup\mathcal{G}_1\cup\mathcal{H}$, where $\mathcal{G}_0$ consists of operations on $A_0$, $\mathcal{G}_1$ consists of operations on $A_1$, and $\mathcal{H}$ consists only of unary operations from $A_1$ into $A_0$ and is assumed to be nonempty. We will also denote the algebra $(A_0, \mathcal{G}_0)$ by $\mathcal{A}_0$ and the algebra $(A_1, \mathcal{G}_1)$ by $\mathcal{A}_1$.
\end{assp}

Situations when $\mathcal{H}$ is empty can be found in Theorem 5.2 (1) of \cite{teohteh}. It states that, for each $i=0, 2$ and each $\vec{e}\in\Omega_i$, the algebra $(A_0, A_1, \mathcal{G}_0, \mathcal{G}_1)$ is an $\vec{e}$-Ramsey algebra if and only if $\mathcal{A}_i$ is a Ramsey algebra. This leads to the following two corollaries.

\begin{cor}
Let $n\in\mathbb{N}$, $\vec{e}\in\Omega$, and let $\mathcal{A}'$ be a reduct of the full matrix algebra not containing the determinant operation. Then $\mathcal{A}'$ is an $\vec{e}$-Ramsey algebra if and only if $\mathcal{G}_{\vec{e}(0)}$ consists of at most one of the two ring operations pertaining to the phylum $A_{\vec{e}(0)}$.
\end{cor}
\begin{proof}
The proof is immediate; we only need to look at the appropriate $\mathcal{A}_0$ or $\mathcal{A}_1$ to decide if $\mathcal{A}'$ or not. Of a priori importance is the fact that any ring of matrices $(\mathcal{M}_n(\mathbb{F}), +, \times)$ is not a Ramsey algebra. This hinges upon the fact that $\mathbb{F}$ can be embedded into $\mathcal{M}_n({\mathbb{F}})$ as diagonal matrices, namely $r\mapsto\diag(r, \ldots, r)$, where $\diag(r, \ldots, r)$ denotes the diagonal matrix whose diagonal elements are all $r$. Diagonal matrices of this form thus form a subalgebra of $(\mathcal{M}_n({\mathbb{F}}), +, \times)$ and, since the subalgebra $(\mathbb{F}, +_\mathbb{F}, \times_\mathbb{F})$ is not a Ramsey algebra, $(\mathcal{M}_n({\mathbb{F}}), +, \times)$ is not a Ramsey algebra either.
\end{proof}

% We now state a simple lemma without giving the proof.

% \begin{lem}\label{onedir}
% Suppose $\mathcal{F}=\mathcal{G}_0\cup\mathcal{G}_1\cup\mathcal{G}$, where $\mathcal{G}$ consists only of operations from Cartesian products of $A_1$ into $A_0$. For each $\vec{e}\in\Omega_1$, each $\vec{e}$-sorted $\vec{b}$, and each $\vec{e}$-sorted $\vec{a}$ such that $\vec{a}\leq_\mathcal{F}\vec{b}$, if $\vec{\beta}$ is the subsequence of $\vec{b}$ all of whose terms belong in $A_1$ and $\vec{\alpha}$ is the corresponding subsequence of $\vec{a}$, then we have
% \begin{enumerate}
% \item $\vec{\alpha}\leq_{\mathcal{G}_1}\vec{\beta}$ and
% \item $\fr_\mathcal{F}^{\vec{e}}(\vec{a})=\fr_{\mathcal{G}_1}^{\vec{e}}(\vec{a})=\fr_{\mathcal{G}_1}(\vec{\alpha})$.
% \end{enumerate}
% \end{lem}

% The lemma above holds because we are only concerned with terms lying in $A_1$ and there is no heterogeneous operation that sends elements to the set $A_1$.

We have, therefore, identified the Ramsey algebraic properties of all reducts of the full matrix algebra for which the determinant operation is absent.

As per Assumption 1, the algebras of concern are such that any heterogeneous operations are unary from $A_1$ into $A_0$. This condition allows us to derive the following theorem.

\begin{thm}\label{big}
For each $\vec{e}\in\Omega_1$, $\mathcal{A}$ is an $\vec{e}$-Ramsey algebra if and only if $\mathcal{A}_1$ is a Ramsey algebra.
\end{thm}
\begin{proof}
($\Rightarrow$) Suppose that $\mathcal{A}$ is an $\vec{e}$-Ramsey algebra and $\vec{\beta}$ is an infinite sequence of $A_1$. Pick any $\vec{e}$-sorted sequence $\vec{b}$ so that $\vec{\beta}$ forms the subsequence of $\vec{b}$ all of whose terms belong in $A_1$.

By hypothesis, let $\vec{a}\leq_\mathcal{F}\vec{b}$ be $\vec{e}$-sorted and homogeneous for $X$. We observe that, if $\vec{\alpha}$ is the subsequence of $\vec{a}$ all of whose terms are members of $A_1$, then (1) $\vec{\alpha}\leq_{\mathcal{G}_1}\vec{\beta}$ and (2) $\vec{\alpha}$ is homogeneous for $X$. This is so because every orderly term over $\mathcal{F}$ with codomain $A_1$ must have as domain a Cartesian power of $A_1$. It then follows that $(A_1, \mathcal{G}_1)$ is a Ramsey algebra.

($\Leftarrow$) Suppose $(A_1, \mathcal{G}_1)$ is a Ramsey algebra. Given $X\subseteq A_1$ and an $\vec{e}$-sorted sequence $\vec{b}$, let $\vec{\beta}$ be the subsequence of $\vec{b}$ consisting of elements of $A_1$. By hypothesis, pick an $\vec{\alpha}\leq_{\mathcal{G}_1}\vec{\beta}$ homogeneous for $X$. In fact, by carefully going through the definition of reduction, we can pick such an $\vec{\alpha}$ so that, for any $\vec{e}$-sorted $\vec{a}$ such that $\vec{\alpha}$ is the subsequence of $\vec{a}$ whose terms are members of $A_1$, we have $\vec{a}\leq_\mathcal{F}\vec{b}$.

Thus, again by the observation we made in the ($\Rightarrow$) case, we have that $\fr_\mathcal{F}^{\vec{e}}(\vec{a})=\fr_{\mathcal{G}_1}(\vec{\alpha})$, whence the homogeneity of $\vec{a}$ for $X$ is established.
\end{proof}

Note that, for $\vec{e}\in\Omega_1$, Theorem \ref{big} offers a complete answer as to when a matrix algebra is an $\vec{e}$-Ramsey algebra:

\begin{cor}
Let $n\in\mathbb{N}$. For any $\vec{e}\in\Omega_1$, all reducts of the full matrix algebra $(\mathcal{M}_n({\mathbb{F}})$, $\mathbb{F}$, $+$, $\times$, $+_{\mathbb{F}}$, $\times_{\mathbb{F}}$, $|\ast|)$ is an $\vec{e}$-Ramsey algebra except for those reducts that keep both matrix operations.
\end{cor}
\begin{proof}
This is because $(\mathcal{M}_n({\mathbb{F}}), +)$ and $(\mathcal{M}_n({\mathbb{F}}), \times)$ are Ramsey algebras (because they are semigroups) and, since $\mathbb{F}$ is embedded in $(\mathcal{M}_n({\mathbb{F}}), +, \times)$, it is not a Ramsey algebra by Theorem \ref{tehring}.
\end{proof}

Thus, we should now focus on the situations when $\vec{e}\in\Omega\setminus\Omega_1$ as well as when $\mathcal{H}$ is nonempty. Henceforth, we make the following assumption:

\begin{assp}
$\mathcal{H}$ is assumed to be a singleton and its sole member will be denoted by $h$ henceforth.
\end{assp}

\begin{thm}\label{firsttheorem}
Suppose that $\vec{e}\in\Omega_0$. If $\mathcal{G}_0=\varnothing$ and $\mathcal{A}_1$ is a Ramsey algebra, then $\mathcal{A}$ is an $\vec{e}$-Ramsey algebra.
\end{thm}
\begin{proof}
Given $X\subseteq A_0$ and any $\vec{e}$-sorted sequence $\vec{b}$, let $\vec{\beta}$ be the subsequence of $\vec{b}$ whose terms are members of $A_1$. By the hypothesis that $(A_1, \mathcal{G}_1)$ is a Ramsey algebra, let $\vec{\alpha}\leq_{\mathcal{G}_1}\vec{\beta}$ be homogeneous for $h^{-1}[X]$.

Using $\vec{\alpha}$, we define the $\vec{e}$-sorted sequence $\vec{a}$ as follows:
\begin{displaymath}
   \vec{a}(i) = \left\{
     \begin{array}{lr}
       \vec{\alpha}(i) & \text{if}\; \vec{e}(i)=1, \\
       h(\vec{\alpha}(i)) & \text{otherwise}.
     \end{array}
   \right.
\end{displaymath}

By way it is defined, we see that $\vec{a}\leq_\mathcal{F}\vec{\alpha}$. Observe also that $\vec{a}\leq_\mathcal{F}\vec{b}$ by the transitivity on the chain of reductions above. Further, $\vec{a}$ is $\vec{e}$-sorted and $\fr^{\vec{e}}_\mathcal{F}(\vec{a})=\{h(\alpha):\alpha\in\fr_{\mathcal{G}_1}(\vec{\alpha})\}$ by the choice of $\vec{a}$ and by Eqv. \ref{Omegafr}. Since $\fr_{\mathcal{G}_1}(\vec{\alpha})\subseteq h^{-1}[X]$ or $\fr_{\mathcal{G}_1}(\vec{\alpha})\subseteq A_1\setminus h^{-1}[X]$, it follows that $\fr^{\vec{e}}_\mathcal{F}(\vec{a})\subseteq X$ or $\fr^{\vec{e}}_\mathcal{F}(\vec{a})\subseteq A_0\setminus X$, respectively, thus proving that $\mathcal{A}$ is an $\vec{e}$-Ramsey algebra.
\end{proof}

\begin{cor}\label{results}
For each $\vec{e}\in\Omega_0$ and $n\in\omega$, the algebras $(\mathcal{M}_n(\mathbb{F}), \mathbb{F}, |\ast|)$, $(\mathcal{M}_n(\mathbb{F}), \mathbb{F}, +, |\ast|)$, and $(\mathcal{M}_n(\mathbb{F}), \mathbb{F}, \times, |\ast|)$ are $\vec{e}$-Ramsey algebras.
\end{cor}
\begin{proof}
Apply the preceding corollary.
\end{proof}

The determinant operation is a homomorphism from the set of matrices equipped with matrix multiplication to the multiplicative group of the underlying field. In general, the notion of a homomorphism can be defined between two algebras of the same signature. Two homogeneous algebras $\mathcal{A}_0$ and $\mathcal{A}_1$ are said to have the same signature if there exists a one-to-one correspondence between $\mathcal{G}_0$ and $\mathcal{G}_1$ such that, if $F\in\mathcal{G}_1$ is an $n$-ary operation, then the corresponding operation $f\in\mathcal{G}_0$ is also $n$-ary. Now, if $\mathcal{G}_0$ and $\mathcal{G}_1$ share the same signature, then an $h:A_1\to A_0$ is a \emph{homomorphism} from $\mathcal{A}_1$ into $\mathcal{A}_0$ if for each corresponding $n$-ary operation $F\in\mathcal{G}_1$ and $f\in\mathcal{G}_0$, and for all $(a_1, \ldots, a_n)$ in the domain of $f$,
\begin{equation}
h(F(a_1, \ldots, a_n))=f(h(a_1), \ldots, h(a_n)).
\end{equation}

\begin{lem}\label{ac}
Suppose that $\vec{e}\in\Omega_0$ and $h$ is a homomorphism from $\mathcal{A}_1$ into $\mathcal{A}_0$. If $\vec{\alpha}\in{^\omega}A_1$ and the $\vec{e}$-sorted sequence $\vec{a}$ are related by
\begin{displaymath}
   \vec{a}(i) = \left\{
     \begin{array}{lr}
       \vec{\alpha}(i) & \text{if}\; \vec{e}(i)=1, \\
       h(\vec{\alpha}(i)) & \text{otherwise},
     \end{array}
   \right.
\end{displaymath}
then, for each $N$-ary $f\in\ot(\mathcal{F})$ having codomain $A_0$ and each $n_1<\cdots<n_N$, there exists an $N$-ary $F\in\ot(\mathcal{G}_1)$ such that $f(\vec{a}(n_1), \ldots, \vec{a}(n_N))=h(F(\vec{\alpha}(n_1), \ldots, \vec{\alpha}(n_N)))$. In addition, $\fr_\mathcal{F}^{\vec{e}}(\vec{a})=\{h(c):c\in\fr_{\mathcal{G}_1}(\vec{\alpha})\}$.
\end{lem}
\begin{proof}
We prove the lemma by induction on the generation of $f$. For the base case, we consider $f\in\mathcal{G}_0\cup\{h\}\cup\{\id_{A_0}\}$. If $f=\id_{A_0}$, then
\begin{equation*}
f(\vec{a}(i))=\id_{A_0}(h(\vec{\alpha}(i)))=h(\id_{A_1}(\vec{\alpha}(i)))
\end{equation*}
which clearly shows that $f(\vec{a}(i))$ is in the stipulated form. For $f=h$, the proof is similar. Thus, suppose now that $f$ is an $N$-ary operation belonging in $\mathcal{G}_0$ and $n_1<\cdots< n_N$. Let $F$ be the corresponding operation in $\mathcal{G}_1$ under the homomorphism $h$. Then
\begin{eqnarray*}
f(\vec{a}(n_1), \ldots, \vec{a}(n_N)) &=& f(h(\vec{\alpha}(n_1)), \ldots, h(\vec{\alpha}(n_N))) \\
&=& h\left(F(\vec{\alpha}(n_1), \ldots, \vec{\alpha}(n_N))\right),
\end{eqnarray*}
which again is in the stipulated form. The last base case to consider is when $f=h$, but this is immediate.

Next, for the inductive step, suppose that $f$ is such that $f(\vec{a}(n_1), \ldots, \vec{a}(n_N))=G(G_1(\tau_1)$, $\ldots$, $G_N(\tau_N))$, where $G\in\mathcal{G}_0\cup\{h\}\cup\{\id_{A_0}\}$ and $\tau_1\ast\cdots\ast\tau_N=\langle\vec{a}(n_1), \ldots, \vec{a}(n_N)\rangle$. For each finite subsequence $\tau=\langle\vec{a}(m_1), \ldots, \vec{a}(m_M)\rangle$ of $\vec{a}$, let $\tilde{\tau}$ denote the finite subsequence $\langle\vec{\alpha}(m_1)$, $\ldots$, $\vec{\alpha}(m_M)\rangle$ of $\vec{\alpha}$. We omit the case when $f=h$ as the proof is immediate. Thus, by induction hypothesis, let $F_1, \ldots, F_N\in\ot(\mathcal{G}_1)$ be such that $G_i(\tau_i)=h(F_i(\tilde{\tau}_i))$ for each $i\in\{1, \ldots, N\}$. Then, denoting by $G'\in\mathcal{G}_1$ the operation corresponding to $G$ under the homomorphism $h$, we now have
\begin{eqnarray*}
f(\vec{a}(n_1), \ldots, \vec{a}(n_N)) &=& G(G_1(\tau_1), \ldots, G_N(\tau_N)) \\
&=& G(h(F_1(\tilde{\tau}_1)), \ldots, h(F_N(\tilde{\tau}_N))) \\
&=& h(G'(F_1(\tilde{\tau}_1), \ldots, F_N(\tilde{\tau}_N))).
\end{eqnarray*}
Since $G', F_1, \ldots, F_N\in\ot(\mathcal{G}_1)$ and $\sigma_1\ast\cdots\ast\sigma_N=\langle\vec{\alpha}(n_1), \ldots, \vec{\alpha}(n_N)\rangle$, it follows that $f(\vec{a}(n_1)$, $\ldots$, $\vec{a}(n_N))$ can be expressed in the stipulated form. This completes the induction proof of the first conclusion of the lemma.

The other conclusion of the lemma can now be deduced easily.
\end{proof}

\begin{thm}\label{heterohomomorphism}
Suppose that $\mathcal{A}_1$ is a Ramsey algebra and $h$ is a homomorphism. Then, $\mathcal{A}$ is an $\vec{e}$-Ramsey algebra for each $\vec{e}\in\Omega_0$.
\end{thm}
\begin{proof}
Suppose that $\vec{e}\in\Omega_0$, $\vec{b}$ is an $\vec{e}$-sorted sequence, and $X\subseteq A_0$. Let $\vec{\beta}$ be the subsequence of $\vec{b}$ whose terms are members of $A_1$. As such, let $\vec{\alpha}\leq_{\mathcal{G}_1}\vec{\beta}$ be homogeneous for $h^{-1}[X]$. Take note that the relation $\vec{\alpha}\leq_\mathcal{F}\vec{\beta}\leq_\mathcal{F}\vec{b}$ holds by the transitivity of $\leq_\mathcal{F}$ and the fact that $\mathcal{G}_1\subseteq\mathcal{F}$.

We now define an $\vec{e}$-sorted sequence $\vec{a}$ by letting
\begin{displaymath}
   \vec{a}(i) = \left\{
     \begin{array}{lr}
       \vec{\alpha}(i) & \text{if}\; \vec{e}(i)=1, \\
       h(\vec{\alpha}(i)) & \text{otherwise}
     \end{array}
   \right.
\end{displaymath}
and we note that $\vec{a}\leq_\mathcal{F}\vec{\alpha}$, hence $\vec{a}\leq_\mathcal{F}\vec{b}$ by the transitivity of $\leq_\mathcal{F}$. We may now apply Lemma \ref{ac}. Namely, each member of $\fr_\mathcal{F}^{\vec{e}}(\vec{a})$ is the image of some $c\in\fr_{\mathcal{G}_1}(\vec{\alpha})$ under $h$ since $h$ is a homomorphism. From this, we conclude that $\fr_\mathcal{F}^{\vec{e}}(\vec{a})\subseteq X$ or $\fr_\mathcal{F}^{\vec{e}}(\vec{a})\subseteq A_0\setminus X$ depending respectively on whether $\fr_{\mathcal{G}_1}(\vec{\alpha})\subseteq h^{-1}[X]$ or $\fr_{\mathcal{G}_1}(\vec{\alpha})\subseteq A_1\setminus h^{-1}[X]$. This is a statement about the homogeneity of $\vec{a}$ for $X$, hence $\mathcal{A}$ is an $\vec{e}$-Ramsey algebra.
\end{proof}

%\begin{proof}
%Assume that $\vec{e}\in\Omega_0$ and let $\vec{b}$ be $\vec{e}$-sorted and $X\subseteq A_0$. Let $\vec{a}\leq_\mathcal{F}\vec{b}$ be $\vec{e}$-sorted such that $\fr_{\mathcal{G}_1}(\vec{a})\subseteq\Sigma\cap A_1'$.
%
%Now, let $\vec{a}'\leq_\mathcal{F}\vec{a}$ be homogeneous for $Y=\{\alpha\in A_1:h(\alpha)\in X\}$, i.e. either
%\begin{equation}\label{frY}
%\fr_\mathcal{F}(\vec{a}')=\fr_{\mathcal{G}_1}(\vec{a}')\subseteq Y\;\text{or}\;\fr_\mathcal{F}(\vec{a}')=\fr_{\mathcal{G}_1}(\vec{a}')\subseteq A_1\setminus Y.
%\end{equation}
%
%We will show that the $\vec{e}$-sorted sequence defined by
%\begin{displaymath}
%   \vec{c}(i) = \left\{
%     \begin{array}{lr}
%       \vec{a}'(i) & \text{if}\; \vec{e}(i)=1, \\
%       h(\vec{a}'(i)) & \text{otherwise},
%     \end{array}
%   \right.
%\end{displaymath}
%is homogeneous for $X$ by appealing to Lemma \ref{ac} and the definition of $Y$.
%
%Noting that $\vec{a}\leq_\mathcal{F}\vec{b}$, we arrive at the conclusion that $\mathcal{A}$ is an $\vec{e}$-sorted Ramsey algebra.
%\end{proof}

\begin{cor}
Suppose $\mathcal{A}_1$ is a Ramsey algebra and $\mathcal{H}$ is a singleton whose member is a homomorphism. Then, $\mathcal{A}$ is an $\vec{e}$-Ramsey algebra for all $\vec{e}\in\Omega_0$.
\end{cor}

\begin{cor}\label{simplemult}
For each $\vec{e}\in\Omega_0$ and $n\in\omega$, $(\mathcal{M}_n(\mathbb{F}), \mathbb{F}, \times_{\mathbb{F}}, \times, |\ast|)$ is an $\vec{e}$-Ramsey algebra.
\end{cor}
\begin{proof}
The determinant operator $|\ast|$ is a homomorphism and $(\mathcal{M}_n, \times)$ is a Ramsey algebra.
\end{proof}

\begin{thm}
Suppose that $\mathcal{A}_1$ is not a Ramsey algebra as witnessed by $\vec{\beta}$ and $X_1$ and that $h$ is \emph{one-to-one} on $\fr_{\mathcal{G}_1}(\vec{\beta})$. Then, $\mathcal{A}$ is not an $\vec{e}$-Ramsey algebra for all nonconstant $\vec{e}\in\Omega_0$.
\end{thm}
\begin{proof}
Let $\vec{e}\in\Omega_0$ be nonconstant. Begin by defining an $\vec{e}$-sorted sequence $\vec{b}$ as follows:
\begin{displaymath}
   \vec{b}(i) = \left\{
     \begin{array}{lr}
       \vec{\beta}(i) & \text{if}\; \vec{e}(i)=1, \\
       h(\vec{\beta}(i)) & \text{otherwise}.
     \end{array}
   \right.
\end{displaymath}
Observe that, for every $\vec{e}$-sorted $\vec{a}\leq_\mathcal{F}\vec{b}$, the subsequence $\vec{\alpha}$ consisting of terms belonging in $A_1$ is such that $\vec{\alpha}\leq_{\mathcal{G}_1}\vec{\beta}$ and, therefore, $\fr_{\mathcal{G}_1}(\vec{\alpha})\not\subseteq X$ and $\fr_{\mathcal{G}_1}(\vec{\alpha})\not\subseteq A_1\setminus X$ by hypothesis. Consequently, since $\vec{e}$ is not constant, we see that $\fr_\mathcal{F}^{\vec{e}}(\vec{a})$ is such that $\fr_\mathcal{F}^{\vec{e}}(\vec{a})\not\subseteq h(X\cap\fr_{\mathcal{G}_1}(\vec{\beta}))$ and $\fr_\mathcal{F}^{\vec{e}}\not\subseteq A_0\setminus h(X\cap\fr_{\mathcal{G}_1}(\vec{\beta}))$ owing to the fact that $h$ is a one-to-one function. Hence, $\mathcal{A}$ is not an $\vec{e}$-Ramsey algebra.
\end{proof}

Note that the theorem above requires that $\vec{e}$ being nonconstant. If $\vec{e}$ is constant, then the conclusion depends solely on whether $\mathcal{A}_0$ is a Ramsey algebra or not.

\begin{cor}
Neither $(\mathcal{M}_n, \mathbb{F}, +, \times, |\ast|)$, $(\mathcal{M}_n, \mathbb{F}, +, \times, +_\mathbb{F}, |\ast|)$, $(\mathcal{M}_n, \mathbb{F}, +, \times, \times_\mathbb{F}, |\ast|)$, nor  the full matrix algebra is an $\vec{e}$-Ramsey algebra for any nonconstant $\vec{e}\in\Omega_0$.
\end{cor}
\begin{proof}
(Sketch.) Note that matrix ring has zero divisors, so we cannot apply Theorem \ref{tehring} directly. However, since our field $\mathbb{F}$ of interest are of characteristic $0$, we may take as bad sequence $\vec{b}$ the sequence of diagonal matrices whose diagonal elements are terms of the bad sequence $\vec{\beta}$ witnessing $\mathbb{F}$ not being a Ramsey algebra. For any infinite ring with identity of characteristic zero, an isomorphic copy of the integers is embedded within. The bad sequence witnessing its failure of being a Ramsey algebra can then be taken to be integral and positive. (We will not prove this fact.) Hence $|\ast|$ will be one-to-one on $\fr_{\mathcal{G}_1}(\vec{\beta})$ and the conclusion follows from the theorem above.
\end{proof}

\begin{lem}\label{long}
Suppose that $h$ is a homomorphism, $\vec{\beta}\in{^\omega}\im(h)$, $\vec{e}\in\Omega_0$, and $\vec{b}$ is $\vec{e}$-sorted such that $\vec{\beta}(i)=h(\vec{b}(i))$ if $\vec{e}(i)=1$ and $\vec{b}(i)=\vec{\beta}(i)$ otherwise.
Then we have:
\begin{enumerate}
\item If $f$ is an $N$-ary member of $\ot(\mathcal{F})$ having codomain $A_0$, then for each $n_1<\cdots<n_N$, there exists $F'\in\ot\left(\mathcal{G}_0\right)$ such that $f\left(\vec{b}(n_1), \ldots, \vec{b}(n_N)\right)=F'\left(\vec{\beta}(n_1), \ldots, \vec{\beta}(n_N)\right)$.
\item If $\vec{u}\leq_{\mathcal{F}}\vec{b}$
and $\vec{u}\in{^\omega}A_0$, then $\vec{u}\leq_{\mathcal{G}_0}\vec{\beta}$.
\end{enumerate}
\end{lem}
\begin{proof}
(2) follows easily from (1), so we will only justify (1). Since $\vec{\beta}\in{^\omega}\im(h)$, we can find an $\vec{\alpha}\in{^\omega}A_1$ such that
\begin{displaymath}
   \vec{b}(i) = \left\{
     \begin{array}{lr}
       \vec{\alpha}(i) & \text{if}\; \vec{e}(i)=1, \\
       h(\vec{\alpha}(i)) & \text{otherwise}.
     \end{array}
   \right.
\end{displaymath}
We may thus apply Lemma \ref{ac} to obtain $F\in\ot(\mathcal{G}_1)$ such that $f\left(\vec{b}(n_1), \ldots, \vec{b}(n_N)\right)=h(F(\vec{\alpha}(n_1), \ldots, \vec{\alpha}(n_N))$. Letting $F'\in\ot(\mathcal{G}_0)$ denote the corresponding operation of $F$ under the homomorphism $h$, we then have
\begin{eqnarray*}
f\left(\vec{b}(n_1), \ldots, \vec{b}(n_N)\right) &=& h(F(\vec{\alpha}(n_1), \ldots, \vec{\alpha}(n_N)) \\
&=& F'(h(\vec{\alpha}(n_1)), \ldots, h(\vec{\alpha}(n_N))) \\
&=& F'(\vec{\beta}(n_1), \ldots, \vec{\beta}(n_N))
\end{eqnarray*}
as desired.
\end{proof}

\begin{thm}\label{3.5}
Suppose that $h$ is a homomorphism and suppose that $\mathcal{A}_0$ is not a Ramsey algebra. If there exists a sequence $\vec{\beta}\in{^\omega}\im(h)$ and a set $X\subseteq A_0$ witnessing the failure of $\mathcal{A}_0$ being a Ramsey algebra, then $\mathcal{A}$ is not an $\vec{e}$-Ramsey algebra for all $\vec{e}\in\Omega_0$ .
\end{thm}
\begin{proof}
Suppose that $\vec{\beta}$ and $X$ are as stipulated in the statement of the theorem and let $\vec{e}\in\Omega_0$ be arbitrary. Choose any sequence $\vec{b}$ as given in Lemma \ref{long} above. Now, suppose $\vec{a}\leq_\mathcal{F}\vec{b}$ is $\vec{e}$-sorted. We want to show that $\vec{a}$ is not homogeneous for $X$, i.e. $\fr_\mathcal{F}^{\vec{e}}(\vec{a})\cap X\neq\varnothing$ and $\fr_\mathcal{F}^{\vec{e}}(\vec{a})\cap(A_0\setminus X)\neq\varnothing$.

Hence, suppose $\vec{u}$ is the subsequence of $\vec{a}$ whose terms are members of $A_0$, then we know from Part 2 of Lemma \ref{long} that $\vec{u}\leq_{\mathcal{G}_0}\vec{\beta}$. By the choice of $\vec{\beta}$ and $X$, we then have $\fr_{\mathcal{G}_0}(\vec{u})\cap X\neq\varnothing$ and $\fr_{\mathcal{G}_0}(\vec{u})\cap(A_0\setminus X)\neq\varnothing$. This implies that $\fr_\mathcal{F}^{\vec{e}}(\vec{a})\cap X\neq\varnothing$ and $\fr_\mathcal{F}^{\vec{e}}(\vec{a})\cap(A_0\setminus X)\neq\varnothing$. Thus, $\vec{a}$ is not homogeneous for $X$ and so $\mathcal{A}$ is not an $\vec{e}$-Ramsey algebra.
\end{proof}

\begin{cor}
$\left(\mathcal{M}_n, \mathbb{F}, +_{\mathbb{F}}, \times_{\mathbb{F}}, |\ast|\right)$ and $\left(\mathcal{M}_n, \mathbb{F}, \times, +_{\mathbb{F}}, \times_{\mathbb{F}}, |\ast|\right)$ are not $\vec{e}$-Ramsey algebras for any $\vec{e}\in\Omega_0$.
\end{cor}
\begin{proof}
This follows from Theorem \ref{3.5} and the observation that the range of the determination operation is $\mathbb{F}$.
\end{proof}

\section{More on Matrix Algebras}
In this section, we tackle the remaining matrix algebras obtainable as reducts of the full matrix algebra $(\mathcal{M}_n(\mathbb{F})$, $\mathbb{F}$, $+, \times, +_{\mathbb{F}}$, $\times_{\mathbb{F}}$, $|\ast|)$. Results in this section are of a negative nature. Note again that the case with $\vec{e}\in\Omega_1$ follows from Theorem \ref{big}.

The first algebras we will look at are $\left(\mathcal{M}_n(\mathbb{F}), \mathbb{F}, \times, +_\mathbb{F}, |\ast|\right)$ and $\left(\mathcal{M}_n(\mathbb{F}), \mathbb{F}, +_\mathbb{F}, |\ast|\right)$.

\begin{thm}
The algebra $\left(\mathcal{M}_n(\mathbb{F}), \mathbb{F}, \times, +_\mathbb{F}, |\ast|\right)$ is not an $\vec{e}$-Ramsey algebra for every nonconstant $\vec{e}\in\Omega_0$.
\end{thm}
\begin{proof}
We define $G_i$ for each $i\in\omega$ to be the diagonal matrix all of whose diagonal entries are $1$ except for its upper left-hand diagonal element, which is given by $5i+1$. Given $\vec{e}$, let $\vec{b}$ be $\vec{e}$-sorted such that
\begin{displaymath}
   \vec{b}(i) = \left\{
     \begin{array}{lr}
       G_i & \text{if}\; \vec{e}(i)=1, \\
       1 & \text{otherwise}.
     \end{array}
   \right.
\end{displaymath}

If $\vec{a}\leq_\mathcal{F}\vec{b}$ is $\vec{e}$-sorted, then every matrix term of $\vec{a}$ is of the form $G_{i_1}\cdots G_{i_M}$ for some $i_1<\cdots<i_M$. Therefore, the quantities $|G_{i_1}\cdots G_{i_M}|=(5i_1+1)\cdots(5i_M+1)\equiv 1$ (mod 5) and $|G_{i_1}\cdots G_{i_M}|+|G_{j_1}\cdots G_{j_N}|=[(5i_1+1)\cdots(5i_M+1)]+[(5j_1+1)\cdots(5j_N+1)]\equiv 2$ (mod 5), where $i_1<\cdots<i_M<j_1<\cdots<j_N$, are both members of $\fr_\mathcal{F}^{\vec{e}}(\vec{a})$ by Eqv.\:\ref{Omegafr}.

Now, define $X\subseteq\mathbb{F}$ by $X=\{r\in\omega:r\equiv 1\;(\text{mod}\;5)\}$, from which we easily see that there is no reduction of $\vec{b}$ homogeneous for $X$. Thus, the desired result follows.
\end{proof}

\begin{thm}
The algebra $\vec{e}\in\Omega_0$ $\left(\mathcal{M}_n(\mathbb{F}), \mathbb{F}, +_\mathbb{F}, |\ast|\right)$ is not an $\vec{e}$-Ramsey algebra for every nonconstant $\vec{e}\in\Omega_0$.
\end{thm}
\begin{proof}
In the absence of matrix operations, the $M, N$ above can be taken to be $1$.
\end{proof}

In the next few theorems, we will make use of the uniqueness of binary representation of the natural numbers (UBR). In most cases, we will have the opportunity to appeal to this uniqueness when arguing about the exponents as well as the bases in the quantities involved. Theorem \ref{tehring} will also play a major role here.

\begin{thm}\label{firstD_i}
The algebras $\mathcal{A}=(\mathcal{M}_n(\mathbb{F})$, $\mathbb{F}$, $+, \times_\mathbb{F}, |\ast|)$ and $\mathcal{A}'=(\mathcal{M}_n(\mathbb{F})$, $\mathbb{F}$, $+, +_{\mathbb{F}}, \times_{\mathbb{F}}, |\ast|)$are not $\vec{e}$-Ramsey algebras for every nonconstant $\vec{e}\in\Omega_0$.
\end{thm}
\begin{proof}
Note that both the algebras contain matrix addition $+$ and field multiplication $\times_{\mathbb{F}}$. We will make use of these operations to show that neither algebra is an $\vec{e}$-Ramsey algebra. Now, let $\vec{e}$ be given and let $\vec{b}$ be an $\vec{e}$-sorted sequence defined by
\begin{displaymath}
   \vec{b}(i) = \left\{
     \begin{array}{lr}
       D_i & \text{if}\; \vec{e}(i)=1, \\
       1 & \text{otherwise},
     \end{array}
   \right.
\end{displaymath}
where $D_i$, the diagonal matrix all of whose diagonal elements are $2^{2^i}$.

Consider $\vec{a}$ an arbitrary $\vec{e}$-sorted reduction of $\vec{b}$. Then, every matrix term of $\vec{a}$ is of the form $D_{i_1}+\cdots+D_{i_M}$ for some $M\in\omega$ and some natural numbers $i_1<\cdots<i_M$. Therefore, taking the determinant of this sum, we obtain
\begin{equation}\label{odd}
\left|D_{i_1}+\cdots+D_{i_M}\right|=\left(2^{2^{i_1}}+\cdots+2^{2^{i_M}}\right)^n,
\end{equation}
which is a member of $\fr_\mathcal{F}^{\vec{e}}(\vec{a})$ by Eqv.\:\ref{Omegafr} since $\vec{e}\in\Omega$.

Choosing another matrix term $D_{j_1}+\cdots+D_{j_N}$ of $\vec{a}$ such that $i_1<\cdots<i_M<j_1<\cdots<j_N$, we have
\begin{equation}\label{even}
\left(2^{u_{i_1}}+\cdots+2^{u_{i_M}}\right)^n\left(2^{u_{j_1}}+\cdots+2^{u_{j_N}}\right)^n=\left(\sum_{(p, q)\in\{1, \ldots, M\}\times\{1, \ldots, N\}}2^{u_{i_p}+u_{j_q}}\right)^n
\end{equation}
which is also a member of $\fr_\mathcal{F}^{\vec{e}}(\vec{a})$ for the same reason above.

Note that as the $i$'s in Eq.\:\ref{odd} are \emph{strictly} increasing, $\{2^{i_p}:p=1, \ldots, M\}$ is a set of pairwise distinct numbers for each $M\in\omega$. Therefore, the equation explicitly expresses $\left|D_{i_1}+\cdots+D_{i_M}\right|$ as the $n$th power of a number $\kappa$ in its binary form. Define the set $Y$ by
\begin{equation*}
Y=\left\{\kappa\in\omega:\kappa=2^{2^{i_1}}+\cdots+2^{2^{i_M}}\;\text{for some}\;M\in\omega\;\text{and some natural numbers}\;i_1<\cdots<i_M\right\}
\end{equation*}
and define
\begin{equation}
X=\left\{\kappa^n:\kappa\in Y\right\}.
\end{equation}
Note that the quantity $\displaystyle\sum_{(p, q)\in\{1, \ldots, M\}\times\{1, \ldots, N\}}2^{u_{i_p}+u_{j_q}}$ appearing on the right hand side of Eq. \ref{even} is not a member of $Y$ due to UBR. As such, we see that the quantity given by Eq. \ref{even} is not a member of $X$, whereas the quantity given by Eq. \ref{odd} is clearly a member of $X$. This demonstrates the nonhomogeneity of $\vec{a}\leq_\mathcal{F}\vec{b}$ for $X$. Therefore, no reduction of $\vec{b}$ can be homogeneous for $X$, hence $\mathcal{A}$ and $\mathcal{A}'$ are not $\vec{e}$-Ramsey algebras if $\vec{e}\in\Omega_0$ and is nonconstant.
\end{proof}

\begin{thm}
The algebra $\mathcal{A}=\left(\mathcal{M}_n(\mathbb{F}), \mathbb{F}, \times_\mathbb{F}, |\ast|\right)$ is not an $\vec{e}$-Ramsey algebra for every nonconstant $\vec{e}\in\Omega_0$.
\end{thm}
\begin{proof}
Let $\vec{b}$, $X\subseteq\mathbb{F}$, and $D_i, i\in\omega$ be as defined in the proof of the preceding theorem.  With the matrix part being an empty algebra, i.e. $\mathcal{G}_1=\varnothing$, every matrix term of an $\vec{e}$-sorted $\vec{a}\leq_\mathcal{F}\vec{b}$ is just a $D_i$ for some $i>0$. In the present case then, we have $N=N'=1$ and the reasoning in the proof of the preceding theorem applies.
\end{proof}

The next theorem requires a lemma.

\begin{lem}\label{pythagorean}
Let $0<i_1<\cdots<i_L$ and $0<j_1<\cdots<j_M<k_1<\cdots<k_N$ be integers. Then,
\begin{equation}
\left(2^{2^{i_1}}+\cdots+2^{2^{i_L}}\right)^2\neq\left(2^{2^{j_1}}+\cdots+2^{2^{j_M}}\right)^2+\left(2^{2^{k_1}}+\cdots+2^{2^{k_N}}\right)^2. \label{ineq1}
\end{equation}
\end{lem}
\begin{proof}
Expanding the left hand side of Inequality \ref{ineq1}, we obtain
\begin{eqnarray}
  \sum_{(p, q)\in\{1, \ldots, L\}^2}2^{2^{i_p}+2^{i_q}} &=& \sum_{p\in\{1, \ldots, L\}}2^{2(2^{i_p})}+\sum_{p, q\in\{1, \ldots, L\}, p\neq q}2\cdot2^{2^{i_p}+2^{i_q}} \nonumber \\
  &=& \sum_{p\in\{1, \ldots, L\}}2^{2^{i_p+1}}+\sum_{p, q\in\{1, \ldots, L\}, p\neq q}2^{2^{i_p}+2^{i_q}+1}. \label{longsum1}
\end{eqnarray}
Call this quantity, a positive integer, $N_1$. Similar expansion of the right hand side of Inequality \ref{ineq1} gives us
\begin{equation*}
\left[\sum_{r\in\{1, \ldots, M\}}2^{2^{j_r+1}}+\sum_{r, s\in\{1, \ldots, M\}, r\neq s}2^{2^{j_r}+2^{j_s}+1}\right]+\left[\sum_{u\in\{1, \ldots, N\}}2^{2^{k_u+1}}+\sum_{u, v\in\{1, \ldots, N\}, u\neq v}2^{2^{k_u}+2^{k_v}+1}\right]
\end{equation*}
\begin{eqnarray}
    =\sum_{r\in\{1, \ldots, M\}}2^{2^{j_r+1}} &+& \sum_{u\in\{1, \ldots, N\}}2^{2^{k_u+1}} \nonumber \\
    &+& \sum_{r, s\in\{1, \ldots, M\}, r\neq s}2^{2^{j_r}+2^{j_s}+1}+\sum_{u, v\in\{1, \ldots, N\}, u\neq v}2^{2^{k_u}+2^{k_v}+1}. \label{longsum2}
\end{eqnarray}
Call this integer $N_2$.

Now, it is crucial we note that the exponents in the sum of Eq. \ref{longsum1} are pairwise distinct, hence Eq. \ref{longsum1} is the binary expansion of $N_1$. Similarly, Eq. \ref{longsum2} is the binary expansion of $N_2$. Thus, if $N_1$ is to equal $N_2$, they must have the exact same binary expansion. In the case $L^2\neq M^2+N^2$, clearly $N_1\neq N_2$ because the total number of summands in Eq. \ref{longsum1} is $L^2$ while the total number of summands in Eq. \ref{longsum2} is $M^2+N^2$. Thus, in such a case, we have that $N_1\neq N_2$ as desired.

On the other hand, suppose that $L^2=M^2+N^2$. We compare the sums $\sum_{p\in\{1, \ldots, L\}}2^{2^{i_p+1}}$ and $\sum_{r\in\{1, \ldots, M\}}2^{2^{j_r+1}}+\sum_{u\in\{1, \ldots, N\}}2^{2^{k_u+1}}$. This is because the exponent for each term of either sums has as binary representation \emph{exactly one} $1$'s appearing while other terms have three, hence, for $N_1$ to equal $N_2$, it is required by UBR that both sums are a fortiori equal. However, such a requirement leads to $L=M+N$, which runs into contradiction with $L^2=M^2+N^2$ since neither $M$ or $N$ is $0$. This shows conclusively that Inequality \ref{ineq1} always holds.
\end{proof}

\begin{thm}
\begin{enumerate}
\item If $n=1$, then the algebra $\left(\mathcal{M}_n(\mathbb{F}), \mathbb{F}, +, +_\mathbb{F}, |\ast|\right)$ is an $\vec{e}$-Ramsey algebra for every $\vec{e}\in\Omega_0$.
\item If $n>1$, then the algebra $\left(\mathcal{M}_n(\mathbb{F}), \mathbb{F}, +, +_\mathbb{F}, |\ast|\right)$ is not an $\vec{e}$-Ramsey algebra for every nonconstant $\vec{e}\in\Omega_0$.
\end{enumerate}
\end{thm}
\begin{proof}
\begin{enumerate}
\item In this case $n=1$, matrices and scalars are essentially the same. If $M$ is a $1\times 1$ matrix, let us denote the entry by $M^\#$; if $X$ is a set of scalars, let us denote the set of corresponding matrices whose entries are in $X$ by $[X]$. We see that $M^\#\in X$ if and only if $M\in [X]$. Note that $|M|=M^\#$ for every $1\times 1$ matrix.

Now, given an $X\subseteq\mathbb{F}$ and an $\vec{e}$-sorted $\vec{b}$, let $\vec{\beta}$ be the subsequence of $\vec{b}$ all of whose terms are matrices and let $\vec{\alpha}\leq_{\{+\}}\vec{\beta}$ be homogeneous for $[X]$. The sequence
\begin{displaymath}
   \vec{a}(i) = \left\{
     \begin{array}{lr}
       \vec{\alpha}^\# & \text{if}\; \vec{e}(i)=0, \\
       \vec{\alpha} & \text{otherwise}
     \end{array}
   \right.
\end{displaymath}
is thus a reduction of $\vec{b}$, i.e. $\vec{a}\leq_\mathcal{F}\vec{b}$. In addition, the homogeneity of $\vec{\alpha}$ for $[X]$ ensures that $\vec{a}$ is homogeneous for $X$.
\item For each $i\in\omega$, let $D_i$ mean the same thing as in the proof of Theorem \ref{firstD_i} and, given any $\vec{e}\in\Omega_0$, let
\begin{displaymath}
   \vec{b}(i) = \left\{
     \begin{array}{lr}
       D_i & \text{if}\; \vec{e}(i)=1, \\
       1 & \text{otherwise}.
     \end{array}
   \right.
\end{displaymath}
Note that, if $\vec{a}\leq_\mathcal{F}\vec{b}$ is $\vec{e}$-sorted, then each matrix term of $\vec{a}$ is of the form $D_{i_1}+\cdots+D_{i_M}$ for some $i_1<\cdots<i_M$ and the determinant of such a term is given by $\left(2^{2^{i_1}}+\cdots+2^{2^{i_M}}\right)^n$, which is a member of $\fr_\mathcal{F}^{\vec{e}}(\vec{a})$ as we appeal to Eq.\:\ref{Omegafr}.

Therefore, define $X\subseteq\mathbb{F}$ by $X=\{r\in\omega:\Theta(r)\}$, where $\Theta(r)$ is the statement $r=\left(2^{2^{i_1}}+\cdots+2^{2^{i_M}}\right)^n$ for some $i_1<\cdots<i_M$. Thus, note that, for each $\vec{a}\leq_\mathcal{F}\vec{b}$ that is $\vec{e}$-sorted, the intersection $X\cap\fr_\mathcal{F}^{\vec{e}}(\vec{a})$ is \emph{nonempty}.

Now, let $k_1<k_2$ be such that $\vec{a}(k_1)=D_{i_1}+\cdots+D_{i_M}$ and $\vec{a}(k_2)=D_{j_1}+\cdots+D_{j_N}$. It then follows that the quantity $\left(2^{2^{i_1}}+\cdots+2^{2^{i_M}}\right)^n+\left(2^{2^{j_1}}+\cdots+2^{2^{j_{N}}}\right)^n\in\fr_\mathcal{F}^{\vec{e}}(\vec{a})$ (again by Eq.\:\ref{Omegafr}) is \emph{not} a member of $X$. This is because the quantity cannot be expressed in the form $\left(2^{2^{l_1}}+\cdots+2^{2^{l_P}}\right)^n$ for any natural numbers $l_1<\cdots<l_P$ owing to Fermat's Last Theorem for $n>2$, while for $n=2$, the result is the content of Lemma \ref{pythagorean}.

We have, thus, shown that the algebra is not an $\vec{e}$-Ramsey algebra for any $n>1$ and any nonconstant $\vec{e}\in\Omega_0$.
\end{enumerate}
\end{proof}

\section{Conclusion}
This paper was aimed at further understanding heterogeneous Ramsey algebras. Specifically, we have looked at heterogeneous algebras consisting of two phyla with some ``disjoint'' set of operations and some heterogeneous unary operations mapping members of a phylum to another. Special cases are when the heterogeneous unary operations are homomorphisms. Such algebras are motivated from the various matrix algebras that we have studied and, as corollaries, we derived results pertaining to the matrix algebras of concern.

While the paper has shed more light on the behavior of heterogeneous Ramsey algebras, algebras for which heterogeneous operations are present remain elusive. The only heterogeneous algebras admitting such operations that have been studied are vector spaces and the results can be found in \cite{teohteh}. A combined look at these two works should be a good starting point for further investigation.

\end{document}